\documentclass[12pt, oneside]{amsart}
\usepackage{geometry} % See geometry.pdf to learn the layout options. There are lots.
\usepackage{graphicx}

\usepackage{amssymb}
\usepackage{epstopdf}
\usepackage{amsthm}
\usepackage{subfigure}
\usepackage{xcolor}

\usepackage{color}
\usepackage{tabularx}

\begin{document}

\title{The genus of complete 3-uniform hypergraphs}
\author{Yifan Jing}
\address{%
Department of Mathematics\\
Simon Fraser University\\
Burnaby, BC, Canada}
\email{yifanjing17@gmail.com}

\author{Bojan Mohar}
\address{%
Department of Mathematics\\
Simon Fraser University\\
Burnaby, BC, Canada}
\email{mohar@sfu.ca}

\thanks{B.M.~was supported in part by the NSERC Discovery Grant R611450 (Canada), by the Canada Research Chairs program, and by the Research Project J1-8130 of ARRS (Slovenia).}
\thanks{On leave from IMFM \& FMF, Department of Mathematics, University of Ljubljana.}%

\date{} % Activate to display a given date or no date

\maketitle

\newtheorem{theorem}{Theorem}[section]
\newtheorem{lemma}[theorem]{Lemma}
\newtheorem{definition}[theorem]{Definition}
\newtheorem{proposition}[theorem]{Proposition}
\newtheorem{corollary}[theorem]{Corollary}
\newtheorem{example}[theorem]{Example}
\newtheorem{conjecture}[theorem]{Conjecture}
\def\ng{\widetilde{\mathsf{g}}}
\def\eg{\widehat{\mathsf{g}}}
\def\g{\mathsf{g}}
\def\K{K_n^3}
\def\E{\mathcal{E}}
\newcolumntype{L}[1]{>{\raggedright\arraybackslash}p{#1}}
\newcolumntype{C}[1]{>{\centering\arraybackslash}p{#1}}
\newcolumntype{R}[1]{>{\raggedleft\arraybackslash}p{#1}}

\newcommand{\Case}[2]{\noindent {\bf Case #1:} \emph{#2}}

\begin{abstract}
In 1968, Ringel and Youngs confirmed the last open case of the Heawood Conjecture by determining the genus of every complete graph $K_n$. In this paper, we investigate the minimum genus embeddings of the complete $3$-uniform hypergraphs $\K$.
Embeddings of a hypergraph $H$ are defined as the embeddings of its associated Levi graph $L_H$ with vertex set $V(H)\sqcup E(H)$, in which $v\in V(H)$ and $e\in E(H)$ are adjacent if and only if $v$ and $e$ are incident in $H$. We determine both the orientable and the non-orientable genus of $\K$ when $n$ is even. Moreover, it is shown that the number of non-isomorphic minimum genus embeddings of $\K$ is at least $2^{\frac{1}{4}n^2\log n(1-o(1))}$. The construction in the proof may be of independent interest as a design-type problem.
\end{abstract}

\section{Introduction}

For a simple graph $G$, let $\g(G)$ be the {\em genus} (sometimes we also use the term {\em orientable genus}) of $G$, that is, the minimum $h$ such that $G$ embeds into the orientable surface $\mathbb{S}_h$ of genus $h$, and let $\ng(G)$ be the {\em non-orientable genus} of $G$ which is the minimum $c$ such that $G$ embeds into the non-orientable surface $\mathbb{N}_c$ with crosscap number $c$. (When $G$ is planar, we define $\ng(G)=0$). If orientability of a surface is not a concern, we may consider the {\em Euler genus} of $G$, which is defined as $\eg(G) = \min\{2\g(G),\ng(G)\}$.

By a surface we mean a compact two-dimensional manifold without boundary. We say $G$ is \emph{2-cell embedded} in a surface if each face of $G$ is homeomorphic to an open disk. Youngs \cite{Y} showed that the problem of determining the orientable genus of a connected graph $G$ is the same as determining a $2$-cell embedding of $G$ with minimum genus. The same holds for the non-orientable genus \cite{PPPV}. It was proved by Thomassen \cite{NPC} that the genus problem is NP-complete. For further background on topological graph theory, we refer to \cite{MT}.

One of the basic questions in topological graph theory is to determine the genus of a graph. This task can be complicated even for small graphs and for families of graphs with simple structure. The genus problem for complete graphs became of central importance in connection with the four-colour problem and its generalization to other surfaces.  Heawood \cite{Heawood} generalized the four-colour conjecture to higher genus surfaces in 1890 and proposed what became known as the \emph{Heawood Map-coloring Conjecture}. The problem was open for almost eight decades, and in 1965 this problem was given the place of honor among Tietze's \emph{Famous Problems of Mathematics} \cite{T}. The problem was eventually reduced to the genus computation for complete graphs and was studied in a series of papers. In 1968, Ringel and Youngs \cite{completegraph} announced the final solution of Heawood's Conjecture. The complete proof was presented in the monograph \cite{mapcolor}. Their proof is split in $12$ cases, some of which were slightly simplified later, but for the most complicated cases, no short proofs are known as of today.

\begin{theorem}[Ringel and Youngs \cite{mapcolor}]
If $n\geq3$ then
\[
\g(K_n)=\bigg\lceil\frac{(n-3)(n-4)}{12}\bigg\rceil.
\]
If $n\geq5$ and $n\neq7$, then
\[
\ng(K_n)=\bigg\lceil\frac{(n-3)(n-4)}{6}\bigg\rceil.
\]
\end{theorem}

A natural generalization of genus problems for graphs is the genus of hypergraphs \cite{JSW,W}. The embeddings of a hypergraph $H$ are defined as the embeddings of its associated \emph{Levi graph}.
The genus problems of hypergraphs are tightly related with the genus of bipartite graphs, $2$-complexes, block designs and finite geometry. We refer to \cite{topics,R85,W} for more background.

In this paper, we determine the genus and the non-orientable genus of complete $3$-uniform hypergraphs $\K$ when $n$ is even.

Our main result is the following.

\begin{theorem}\label{thm:main}
If $n\geq4$ is even, then
\[
\g(\K)=\frac{(n-2)(n+3)(n-4)}{24}.
\]
If $n\geq6$ is even, then
\[
\ng(\K)=\frac{(n-2)(n+3)(n-4)}{12}.
\]
\end{theorem}

In our proofs, we construct a set of Eulerian circuits satisfying certain compatibility conditions. The construction may be independent interest as a design type problem.

If $n$ is odd, the genus of $\K$ is strictly greater than $\big\lceil\tfrac{1}{24}(n-2)(n+3)(n-4)\big\rceil$. Surprisingly, it is actually much larger than that. These cases will be dealt with in a separate paper \cite{odd}.

The construction used to prove Theorem \ref{thm:main} has lots of flexibility and it can be generalized to give many non-isomorphic minimum genus embeddings.

\begin{theorem}
If $n$ is even, there exist at least $2^{\frac{1}{4}n^2\log n(1-o(1))}$ non-isomorphic (orientable and non-orientable, respectively) minimum genus embeddings of $\K$, where the logarithm is taken base $2$.
\end{theorem}

Inspired by \cite[Theorem 1.5]{JM}, we also study the genus of hypergraphs with multiple edges. In that result a phase transition occurs when studying the genus of random bipartite graphs with parts of size $n_1$ and $n_2$, where $n_2$ is constant, $n_1\gg 1$, and edge probability is $p = \Theta(n_1^{-1/3})$. The following hypergraph appears in the analysis related to that case.  Let $m\K$ be the $3$-uniform hypergraph such that each triple of vertices is contained in exactly $m$ edges. We have the following result.

\begin{theorem}\label{thm:multi}
Let $m$ be a positive integer and let $n\geq4$ be even. Then
\[
\g(m\K)=\frac{(n-2)(mn(n-1)-12)}{24}
\]
and
\[
\ng(m\K)=\frac{(n-2)(mn(n-1)-12)}{12}.
\]
\end{theorem}

%In a forthcoming work \cite{odd}, we estimate the orientable and non-orientable genus of complete $3$-uniform hypergraphs $\K$ when $n$ is odd, and determine the genus when $n\equiv1,3\mod6$.

The paper is organized as follows. In the next section, we give basic definitions and results in topological graph theory. In addition, we present the main tools used in the proof of the main theorem. In Section 3, we prove Theorem \ref{thm:main}. In Section 4, we show that the number of non-isomorphic minimum genus embeddings constructed in the proof of Theorem \ref{thm:main} is abundant. Section 5 resolves the genus of hypergraphs with multiple edges and contains the proof of Theorem \ref{thm:multi}.

\section{Embeddings of complete 3-uniform hypergraphs}%

We will use standard graph theory definitions and notation as used by Diestel \cite{Diestel}. As previously mentioned, an embedding of a graph $G$ on a surface is a drawing of $G$ on that surface without edge-crossings. Every 2-cell embedding (and thus also any minimum genus embedding) of $G$ can be represented combinatorially by using the corresponding {\em rotation system} $\pi=\{\pi_v\mid v\in V(G)\}$ where a {\em local rotation} $\pi_v$ at the vertex $v$ is a cyclic permutation of the neighbours of $v$. In addition to this, we also add the {\em signature}, which is a mapping $\lambda: E(G)\to\{1,-1\}$ and describes if the local rotations around the endvertices of an edge have been chosen consistently or not. The signature is needed only in the case of non-orientable surfaces; in the orientable case, we may always assume the signature is trivial (all edges have positive signature). The pair $(\pi,\lambda)$ is called the \emph{embedding scheme} for $G$. For more background on topological graph theory, we refer to \cite{GT,MT}.

We say that two embeddings $\phi_1,\phi_2: G\to S$ of a graph $G$ into the same surface $S$ are \emph{equivalent} (or \emph{homeomorphic}) if there exists a homeomorphism $h: S\to S$ such that $\phi_2 = h\phi_1$. By \cite[Corollary 3.3.2]{MT}, (2-cell) embeddings are determined up to equivalence by their embedding scheme $(\pi,\lambda)$, and two such embedding schemes $(\pi,\lambda)$ and $(\pi',\lambda')$ determine equivalent embeddings if and only if they are \emph{switching equivalent}. This means that there is a vertex-set $U\subseteq V(G)$ such that $(\pi',\lambda')$ is obtained from $(\pi,\lambda)$ by replacing $\pi_u$ with $\pi^{-1}_u$ for each $u\in U$ and by replacing $\lambda(e)$ with $-\lambda(e)$ for each edge $e$ with one end in $U$ and the other end in $V(G)\setminus U$.

Let $H$ be a hypergraph. The associated {\em Levi graph} of $H$ is the bipartite graph $L_H$ defined on the vertex set $V(H)\cup E(H)$, in which $v\in V(H)$ and $e\in E(H)$ are adjacent if and only if $v$ and $e$ are incident in $H$; see \cite{C} or \cite{W}. In this paper, we use $K_n^{3}$ to denote the complete $3$-uniform hypergraph of order $n$, and we denote its Levi graph by $L_n$. The vertices of $L_n$ corresponding to $V(K_n^3)=[n]$ will be denoted by $X_n$ and the vertex-set of cardinality $\binom{n}{3}$ corresponding to the edges of $K_n^3$ will be denoted by $Y_n$. Following \cite[Chapter 13]{W}, we define embeddings of a hypergraph $H$ in surfaces as the 2-cell embeddings of its Levi graph $L_H$. We define the {\em genus} $\g(H)$ (the {\em non-orientable genus} $\ng(H)$, and the {\em Euler genus} $\eg(H)$) as the genus (non-orientable genus, and Euler genus, respectively) of $L_H$.

It is easy to see that each embedding of a hypergraph $H$ can be represented by choosing a point in a surface for each vertex of $H$ and a closed disk $D_e$ for each edge $e\in E(H)$ such that for any two edges $e,f$, the intersection of the corresponding disks $D_e\cap D_f$ is precisely the set of points in $e\cap f$. We refer to \cite{W} for more details and to Figure~\ref{fig:K4} for an example.

\begin{figure}[htb]
\centering
\begin{minipage}[t]{0.45\textwidth}
\centering
\includegraphics[width=2in]{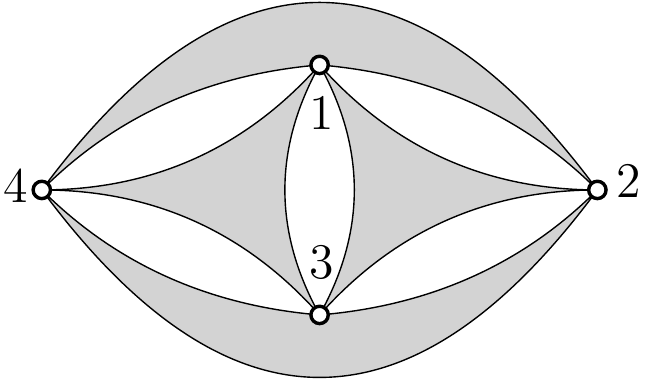}
%\caption{}
\end{minipage}\hfill\begin{minipage}[t]{0.55\textwidth}
\centering
\includegraphics[width=1.8in]{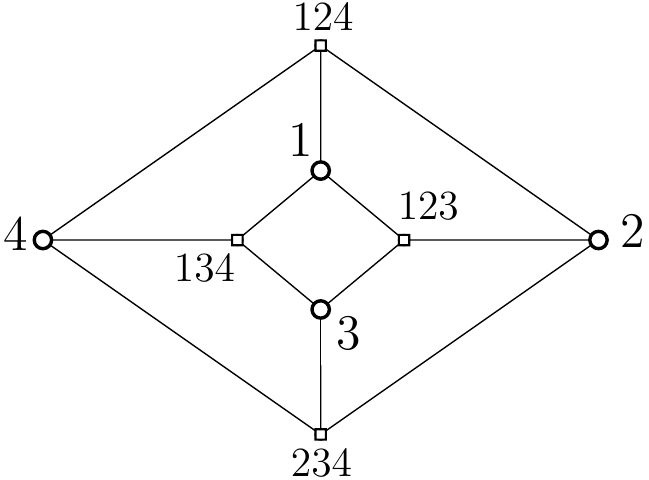}
%\caption{}
\end{minipage}
\caption{Planar embeddings of $K_4^3$ and its Levi graph $L_4$.}\label{fig:K4}
\end{figure}

Since $L_H$ is bipartite, we have the following simple corollary of Euler's Formula (see \cite[Proposition 4.4.4]{MT}).

\begin{lemma}
\label{lem:hypergenus}
Let $H$ be a $3$-uniform hypergraph with $n$ vertices and $e$ edges. Then
\begin{equation}\label{*}
\eg(H)\geq\tfrac{1}{2}e-n+2.
\end{equation}
Moreover, equality holds in {\rm(\ref{*})} if and only if the Levi graph $L_H$ admits a quadrilateral embedding in some surface.
\end{lemma}

In the case of the complete $3$-uniform hypergraphs we obtain:

\begin{proposition}\label{prop:euler}
For every $n\geq4$ we have $\eg(\K)\geq\Big\lceil\frac{(n-2)(n+3)(n-4)}{12}\Big\rceil$.
\end{proposition}

From now on, we assume $n\geq4$ is an even integer. For each $i$ ($1\leq i\leq n$), let $K_n-i$ be the labelled complete graph defined on the vertex set $[n]\setminus \{i\}$. Suppose $T_i$ and $T_i^\prime$ are Eulerian circuits in $K_n-i$. If $T_i^\prime$ is the reverse of $T_i$, we denote it by $T_i^{-1}$ and view them to be equivalent. Two families ${\mathcal F},{\mathcal F'}$ of circuits are \emph{equivalent} if there is a bijection $f:{\mathcal F} \to {\mathcal F'}$ such that for each $C\in {\mathcal F}$ either $f(C)=C$ or $f(C)=C^{-1}$.

Suppose $T_i$ is an Eulerian circuit in $K_n-i$ and $T_j$ in $K_n-j$, where $j\neq i$. Define a {\em transition} through $j$ in $T_i$ as a subtrail of $T_i$ consisting of two consecutive edges $aj$ and $jb$, and we denote it simply by $ajb$ (which may sometimes be written as $a,j,b$). We say that $T_i$ and $T_j$ are {\em compatible} if for every transition $ajb$ in $T_i$, there is a transition $aib$ or $bia$ in $T_j$, and $T_i$ and $T_j$ are {\em strongly compatible} if for every transition $ajb$ in $T_i$, there is the transition $bia$ in $T_j$. Note that this gives a bijective correspondence between $\frac{n-2}{2}$ transitions through $j$ in $T_i$ and $\frac{n-2}{2}$ transitions through $i$ in $T_j$. We call a set of trails $\{T_1,\dots,T_n\}$ an {\em embedding set} if $T_i$ is an Eulerian circuit in $K_n-i$ for each $i=1,\dots,n$ and any two of them are compatible. An embedding set is {\em strong} if for every $i\neq j$, $T_i$ and $T_j$ are strongly compatible. In our construction of embeddings, we will use different rules when specifying Eulerian circuits for odd and even values of $i$, and we will say that $i\in[n]$ is an {\em odd vertex} (or {\em even vertex}) when $i$ is odd (or even) viewed as an integer.

The following result is our main tool in this paper.

\begin{theorem}\label{thm:embed}
Let $n\geq4$ be an even integer. There exists a bijection between equivalence classes of the (labelled) quadrilateral embeddings of the Levi graph $L_n$ of $\K$ and the equivalence classes of embedding sets of size $n$. Under this correspondence, strong embedding sets correspond to orientable quadrilateral embeddings.
\end{theorem}

\begin{proof}
Suppose $\Pi=\{\pi_v\mid v\in V(L_n)\}$ is a quadrilateral embedding of $L_n$. Recall that $X_n=[n]$ and $Y_n=\binom{[n]}{3}$ is the bipartition of $L_n$. For every vertex $i\in X_n$, consider the local rotation $\pi_i$ around $i$. Note that the neighbors of $i$ are all $\binom{n-1}{2} =: N$ triples of elements of $[n]$ which contain $i$, and all of them have degree $3$ in $L_n$.

Each pair of consecutive vertices (triples) in $\pi_i$ determines a 4-face with two vertices in $X_n$, say $i$ and $j$. Then both triples are adjacent to $i$ and to $j$ in $L_n$, so they both contain $i$ and $j$. Let us now consider the two 4-faces containing the edge joining $i$ and a triple $ijk$. Since this triple is adjacent to vertices $j$ and $k$ in $L_n$, one of the neighbors of $i$ preceding or succeeding $ijk$ in the local rotation $\pi_i$ contains $j$ and the other one contains $k$. Therefore there is a sequence $a_1,a_2,\dots,a_N$ such that $a_j$ is the common element between the $j$th and $(j+1)$st neighbor of $i$ in $\pi_i$. Moreover, the $j$th neighbor of $i$ is the triple $ia_{j-1}a_j$ (where $a_0=a_N$). Clearly, the cyclic sequence $T_i = (a_1a_2\dots a_N)$ is an Eulerian circuit in $K_n-i$ since the consecutive pairs $a_{j-1}a_j$ ($1\le j\le N$) run over all pairs in $[n]\setminus \{i\}$.

\begin{figure}[htb]
\centering
\includegraphics[width=11cm]{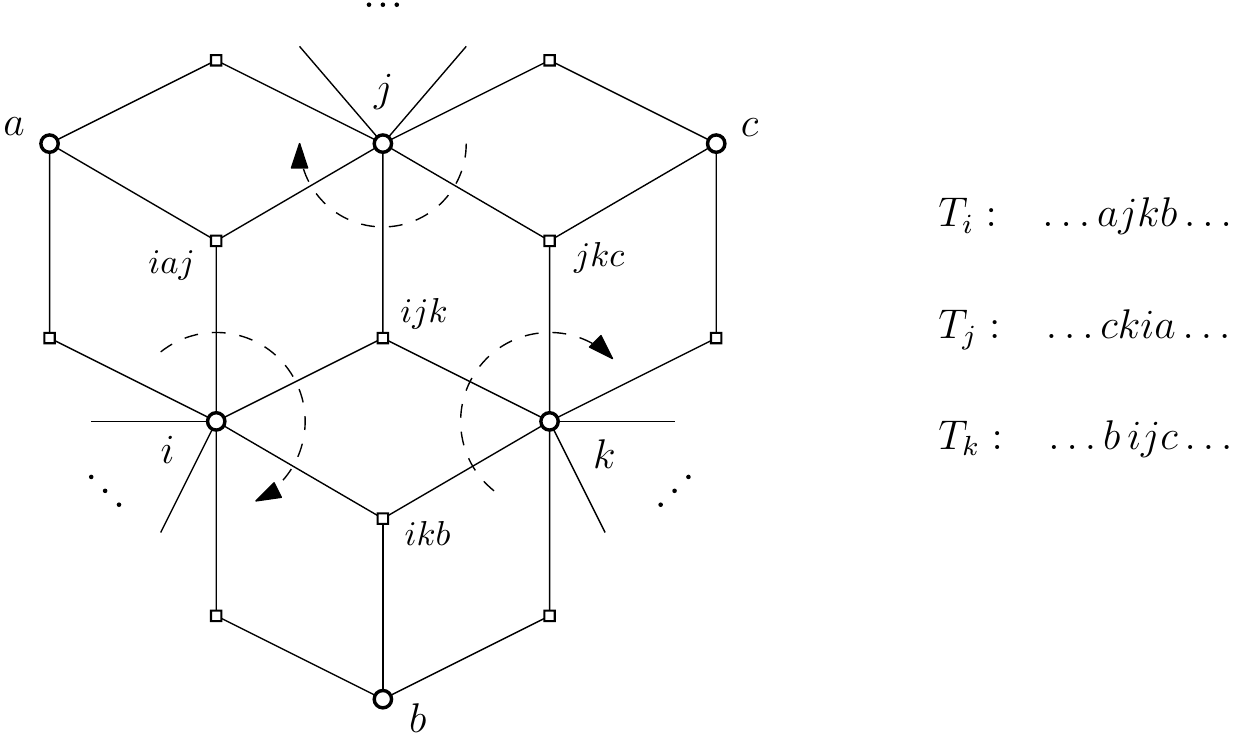}
\caption{A quadrilateral embedding around a vertex $ijk\in Y_n$. The chosen clockwise rotation around vertices $i,j,k$ is indicated by the dashed circular arcs.}\label{fig:locally around ijk}
\end{figure}

Suppose $iaj$ and $ijk$ are consecutive neighbors of the vertex $i$ in $\pi_i$ and assume $iaj\to ijk$ is clockwise. See Figure \ref{fig:locally around ijk} for clarification. That means, $ajk$ is a transition in $T_i$. Now consider the local rotation $\pi_j$. Clearly, $iaj$ and $ijk$ are consecutive vertices in $\pi_j$. Moreover, assuming the local rotations around $i$ and $j$ are chosen consistently with the clockwise orientation in the face containing $i, iaj, j, ijk$, we have $ijk\to iaj$ is clockwise. That means that $T_i$ and $T_j$ are compatible (strongly in the orientable case). Therefore, $\{T_1,\dots,T_n\}$ form an embedding set (or strong embedding set).

This gives a correspondence $(\pi,\lambda) \mapsto \{T_1,\dots,T_n\}$. Let us first observe that equivalent embedding schemes (obtained by switching over a vertex-set $U\subseteq V(L_n)$) correspond to changing the Eulerian circuits $T_u$ with their inverse circuits $T_u^{-1}$ for $u\in U\cap X_n$. Thus the correspondence preserves equivalence.

To see that the described correspondence is injective, consider two quadrilateral embeddings with schemes $\Pi^1 = (\pi^1,\lambda_1)$ and $\Pi^2 = (\pi^2,\lambda_2)$, whose embedding set $\{T_1,\dots,T_n\}$ is the same. This in particular means that $\pi^1$ and $\pi^2$ agree on $X_n$. Clearly, this implies that the set of quadrangular faces is the same for both embeddings (for every $\pi^1_i$-consecutive neighbors $iaj\to ijk$ the corresponding 4-face has vertices $i,iaj,j,ijk$). By \cite[Corollary 3.3.2]{MT}, this implies that $\Pi^1$ and $\Pi^2$ are equivalent.

In order to show the map is surjective, suppose $\mathcal{E}=\{T_1,\dots, T_n\}$ is an embedding set. We have to show that there is a quadrilateral embedding of $K_n^3$ such that this embedding returns an equivalent embedding set under the correspondence described in the first part of the proof. The quadrilateral embedding will be given by an embedding scheme $\Pi=(\pi,\lambda)$ which is determined as follows.

For $i\in X_n$, let $T_i$ be the circuit $a_0a_1a_2\dots a_N$, where $a_0=a_N$. Then we define the rotation $\pi_i$ around the vertex $i$ as the cyclic permutation:
$$
   \pi_i = (ia_0a_1, ia_1a_2, ia_2a_3, \dots, ia_{N-1}a_N ).
$$
For each triple $ijk\in Y_n$ (where $i<j<k$), set $\pi_{ijk}= (i,j,k)$. Finally, define the signature as follows. Given $i<j<k$, let $e_1$, $e_2$, and $e_3$ be the edges joining $ijk$ with the vertex $i$, $j$, and $k$, respectively. We set $\lambda(e_1)=1$ if the edge $jk$ appears in the direction from $j$ to $k$ in $T_i$. Otherwise, set $\lambda(e_1)=-1$. Similarly, set $\lambda(e_2)=1$ ($\lambda(e_3)=1$) if and only if the edge $ki$ ($ij$) appears in $T_j$ ($T_k$) in the direction from $k$ to $i$ (from $i$ to $j$). By these rules it is clear that equivalent embedding sets give equivalent embedding schemes, and that $\Pi$ will give back the same embedding set. It remains to see that the embedding $\Pi$ is quadrilateral. To see this, consider a triple $ijk$ ($i<j<k$) and the faces around it. Figure \ref{fig:locally around ijk} should help us to visualize the situation. By changing the embedding set $\mathcal E$ to an equivalent embedding set (by possibly changing $T_i,T_j,T_k$ to their inverses), we may assume that $T_i$ traverses $jk$ in the direction from $j$ to $k$, $T_j$ traverses $ki$ from $k$ to $i$, and $T_k$ traverses $ij$ from $i$ towards $j$. Then $\lambda(e_1)=\lambda(e_2)=\lambda(e_3)=1$. Let $T_i: \dots ajkb \dots$ and $T_j: \dots ckia\dots$. Here we used compatibility condition to conclude that $kia$ is a transition in $T_j$. Compatibility condition implies that $T_k: \dots bijc \dots$. This implies that the faces around $ijk$ are precisely as shown in the figure. Since $ijk$ was arbitrary, we conclude that all faces are quadrilaterals, which we were to prove.
\end{proof}

\begin{figure}[htb]
\centering
\includegraphics[width=2.35in]{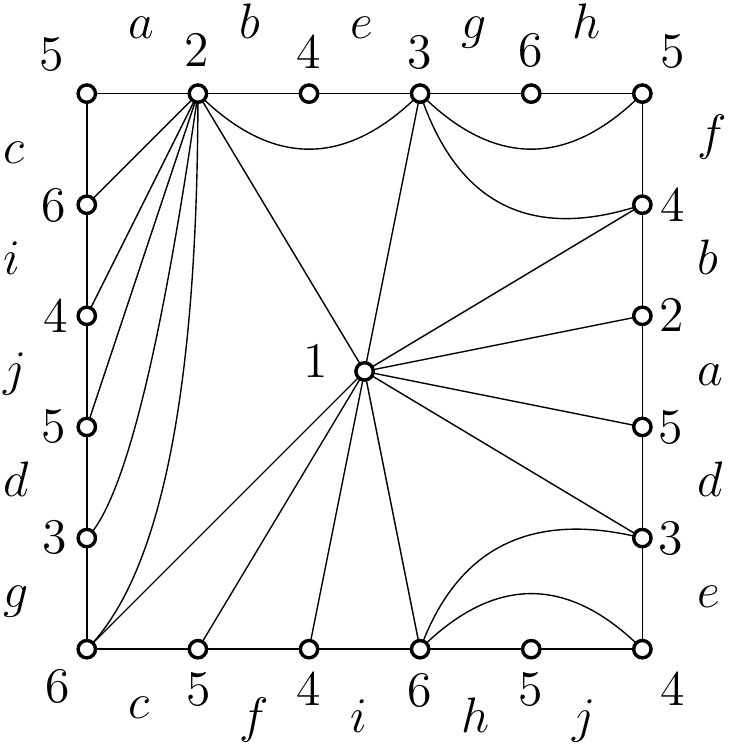}
\caption{An orientable embedding of $K_6^3$ in the triple torus: By putting the vertex $\{i,j,k\}$ inside each triangular face $ijk$ shown in the figure and adding the edges from the vertex to $i,j$ and $k$, we obtain an embedding of the Levi graph $L_6$ with all faces of length $4$.}\label{fig:K6}
\end{figure}

\begin{example}
$\g(K_6^3)=3$.
\end{example}

Figure \ref{fig:K6} shows an embedding of $K_6^3$ in the orientable surface of genus $3$. By identifying the edges on the boundary of the square with the same label $a-h$, we obtain an embedding on the triple torus. Its strong embedding set $\E_6$ is the following:
\begin{align*}
  &T_1: 3,4,2,5,3,6,4,5,6,2; \\
  &T_2: 4,3,1,6,3,5,4,6,5,1; \\
  &T_3: 1,2,4,6,1,5,2,6,5,4; \\
  &T_4: 2,1,3,5,1,6,2,5,6,3; \\
  &T_5: 6,1,4,3,6,4,2,3,1,2; \\
  &T_6: 5,2,4,1,3,4,5,3,2,1.
\end{align*}

\section{Complete $3$-uniform hypergraphs of even order}%section3

In this section, we will construct minimum genus embeddings of $\K$ for every even $n\ge4$.

\begin{theorem}\label{thm:ori}
If $n\geq4$ is even, then
\[
\g(\K)=\frac{(n-2)(n+3)(n-4)}{24}.
\]
\end{theorem}

\begin{proof}
By Lemma \ref{lem:hypergenus} and Proposition \ref{prop:euler} it suffices to show that for even $n\geq4$, the graph $L_n$ has an orientable quadrangular embedding.
%By Proposition \ref{prop:euler}, it suffices to show that $\g(\K)\leq\frac{(n-2)(n+3)(n-4)}{24}$.
We will prove it by induction on $n$. The base case when $n=4$ is clear from Figure \ref{fig:K4}, so we proceed with the induction step.

Assume $L_n$ quadrangulates some orientable surface, and $\mathcal{E}_n=\{T_1,\dots,T_n\}$ is the corresponding strong embedding set, where $T_i$ is an Eulerian circuit in $K_n-i$ ($i\in[n]$). Now we consider $L_{n+2}$ with two new vertices in $X_{n+2}=[n+2]$. For brevity we will write $x=n+1$ and $y=n+2$. For every odd vertex $1\leq i\leq n-1$, $T_i$ contains $\tfrac{n-2}{2}$ transitions of the form $a,i+1,b$. We arbitrarily pick one of those transitions, and denote it by $a_i,i+1,b_i$. In the next step, we are going to insert a trail $E_i$ (defined below) between $i+1$ and $b_i$ in $T_i$ to get an Eulerian circuit in $K_{n+2}-i$. The new, longer circuit will be denoted by $T_i^\prime$. For the trail $T_{i+1}$ in $\E_n$, since $a_i,i+1,b_i$ is a transition in $T_i$, the transition $b_i,i,a_i$ is contained in $T_{i+1}$ by the strong compatibility condition. Similarly as what we do for $T_i$, we insert a trail $E_{i+1}$ between $i$ and $a_i$ in $T_{i+1}$, and the new longer circuit we get is denoted by $T_{i+1}^\prime$.

For every odd vertex $1\leq i\leq n-1$, let $\sigma_i$ be the permutation of the set $[n]\setminus\{ i, i+1\}$ that is obtained from the sequence $1,2,\dots,n$ by removing $i$ and $i+1$ and by switching the pairs $2j-1,2j$ for $j=1,\dots,\tfrac{i-1}{2}$. Specifically:
\begin{align*}
  &\sigma_1 = 3,4,5,6,\dots, n-1, n;        \\
  &\sigma_3 = 2,1,5,6,\dots, n-1, n;          \\
  &\quad\quad \cdots \\
  &\sigma_i =2,1,4,3,\dots,i-1,i-2,i+2,i+3,i+4,\dots,n-1,n; \\
  & \quad\quad\cdots \\
  &\sigma_{n-1} =2,1,4,3,\dots,n-2,n-3.
\end{align*}

We construct $E_i$ as follows. We start with $x$, and then insert $y$ and $x$ consecutively in the interspace of numbers in $\sigma_i$, and add $x,y,i+1$ at the end, for every odd vertex $1\leq i\leq n-1$. For the case $E_{i+1}$, we start with $y$, insert $x$ and $y$ (alternating) in the interspace of numbers in $\sigma_i$, and add $y,x,i$ at the end. To be more precise, we get the following:
\begin{align*}
  &E_1 = x,3,y,4,x,5,y,6,\dots,x, n-1,y, n,x,y,2;        \\
  &E_2 = y,3,x,4,y,5,x,6,\dots,y, n-1,x, n,y,x,1;        \\
    &\quad\quad \cdots \\
  &E_i =x,2,y,1,\dots,x,i-1,y,i-2,x,i+2,y,\dots,x,n-1,y,n,x,y,i+1; \\
  &E_{i+1} =y,2,x,1,\dots,y,i-1,x,i-2,y,i+2,x,\dots,y,n-1,x,n,y,x,i; \\
  & \quad\quad\cdots \\
  &E_{n-1} =x,2,y,1,x,4,y,3,\dots,x,n-2,y,n-3,x,y,n; \\
  &E_{n} =y,2,x,1,y,4,x,3,\dots,y,n-2,x,n-3,y,x,n-1.
\end{align*}

It is easy to see that $T_i^\prime$ and $T_{i+1}^\prime$ are Eulerian circuits in $K_{n+2}-i$ and $K_{n+2}-(i+1)$. To verify the strong compatibility of these Eulerian circuits, note that our construction preserves almost all transitions in $\E_n$, except for every odd $i$ we break the transition $a_i,i+1,b_i$ in $T_i$, and the transition $b_i,i,a_i$ in $T_{i+1}$. That means we only need to check the strong compatibility of transitions in $E_i$. If $j$ and $i$ are both odd and $j<i$, this is true since $x,i,y$ is a transition in $E_j$ and $y,j,x$ is a transition in $E_i$. Similar observations hold in the other three cases depending on the parities of $j$ and $i$. This shows that $T_a^\prime$ and $T_b^\prime$ are strongly compatible for every $1\leq a<b\leq n$.

In the final step, we will define Eulerian circuits $T_x^\prime$ and $T_y^\prime$, such that $\E_{n+2}=\{T_1^\prime,\dots,T_n^\prime,T_x^\prime,T_y^\prime\}$ is a strong embedding set. We have to fix some transitions in $T_x^\prime$ and $T_y^\prime$ in order to get strong compatibility with circuits $T_j^\prime$ ($1\leq j\leq n$). We list these transitions in the following tables, where we assume $3\leq i\leq n-3$ is an odd vertex.

\begin{center}
\begin{tabular}{ c | c | c | c | c  }
\hline
\multicolumn{5}{c}{Transitions in $T_x^\prime$ through odd vertices ($3\leq i\leq n-3$)}\\
\hline
3 1 2 & & 2 $i$ $i+1$ & & 2 $n-1$ $n$ \\
5 1 4 & & 4 $i$ 1 & & 4 $n-1$ 1 \\
7 1 6 & & 6 $i$ 3 & & 6 $n-1$ 3 \\
9 1 8 & $\dots$ & \vdots & $\dots$& \\
 & & $i-1$ $i$ $i-4$ && \\
\vdots & & $i+2$ $i$ $i-2$ & & \vdots \\
 & & $i+4$ $i$ $i+3$ & & \\
 & & \vdots & & $n-4$ $n-1$ $n-7$ \\
$n-1$ 1 $n-2$ & & $n-1$ $i$ $n-2$ & &$n-2$ $n-1$ $n-5$ \\
$y$ $1$ $n$  & & $y$ $i$ $n$& & $y$ $n-1$ $n-3$\\
\hline
\end{tabular}
\end{center}
\begin{center}
\begin{tabular}{ c | c | c | c | c }
\hline
\multicolumn{5}{c}{Transitions in $T_x^\prime$ through even vertices ($4\leq i+1\leq n-2$)}\\
\hline
4 2 3  & & 1 $i+1$ 2 & & 1 $n$ 2 \\
6 2 5  & & 3 $i+1$ 4 & & 3 $n$ 4 \\
8 2 7 & & \vdots & & \\
 &  & $i-2$ $i+1$ $i-1$ & & \\
\vdots & $\dots$ & $i+3$ $i+1$ $i+2$ & $\dots$ & \vdots \\
 &  & \vdots & &  \\
$n$ $2$ $n-1$& & $n$ $i+1$ $n-1$ & & $n-3$ $n$ $n-2$ \\
1 2 $y$ & & $i$ $i+1$ $y$ & & $n-1$ $n$ $y$ \\
 \hline
\end{tabular}
\end{center}

If $T_x^\prime$ has all the transitions listed in the tables above, then it is strongly compatible with $T_j^\prime$ for every $1\leq j\leq n$. Since each pair of two different numbers will consecutively appear in $T_x^\prime$ exactly once, the above tables give us the following $n/2$ subtrails $\{A_1,\dots,A_{\frac{n}{2}}\}$ in $T_x^\prime$ where we also let $3\leq i\leq n-3$ be odd.
\begin{align*}
&A_1=y,1,n,2,n-1,n,y;\\
&\quad\quad\cdots\\
&A_{\frac{i+1}{2}}=y,i,n,i+1,F_i(1), F_i(2),\dots, F_i(\tfrac{i-1}{2}), n-i, n+1-i, y; \\
&\quad\quad\cdots \\
&A_{\frac{n}{2}}=y,F_{n-1}(1), F_{n-1}(2),\dots,F_{n-1}(\tfrac{n}{2}),2,y.
\end{align*}
where $F_i(j)$ ($3\leq i\leq n-3$) is a subtrail of length $4$ such that $F_i(j)=n+1-2j,i-2j,n-2j,i+1-2j$ and $F_{n-1}(j)=n+1-2j$.

Similarly, the following tables list all transitions in $T_y^\prime$ forced by strong compatibility with $T_j^\prime$ for every $1\leq j\leq n$.

\begin{center}
\begin{tabular}{ c | c | c | c | c }
\hline
\multicolumn{5}{c}{Transitions in $T_y^\prime$ through odd vertices ($3\leq i\leq n-3$)}\\
\hline
4 1 3 & & 1 $i$ 2 & & 1 $n-1$ 2 \\
6 1 5  & & 3 $i$ 4 & & 3 $n-1$ 4 \\
8 1 7 & & \vdots & & \\
 & & $i-2$ $i$ $i-1$ & & \\
\vdots  & $\dots$ & $i+3$ $i$ $i+2$ & $\dots$ & \\
 & & \vdots & & \vdots \\
$n$ $1$ $n-1$  & & $n$ $i$ $n-1$ & & $n-3$ $n-1$ $n-2$ \\
2 1 $x$ & & $i+1$ $i$ $x$ & & $n$ $n-1$ $x$ \\
\hline
\end{tabular}
\end{center}
\begin{center}
\begin{tabular}{ c | c | c | c | c }
\hline
\multicolumn{5}{c}{Transitions in $T_y^\prime$ through even vertices ($4\leq i+1\leq n-2$)}\\
\hline
3 2 1 & & 2 $i+1$ $i$ & & 2 $n$ $n-1$ \\
5 2 4 & & 4 $i+1$ 1 &  & 4 $n$ 1 \\
7 2 6 & & 6 $i+1$ 3 & & 6 $n$ 3 \\
9 2 8 & $\dots$ & \vdots & $\dots$ & \\
 & & $i-1$ $i+1$ $i-4$ &  & \\
\vdots& & $i+2$ $i+1$ $i-2$ & & \vdots \\
 & & $i+4$ $i+1$ $i+3$ & & \\
 & & \vdots & &$n-4$ $n-2$ $n-7$ \\
$n-1$ 2 $n-2$ & & $n-1$ $i+1$ $n-2$ & & $n-2$ $n$ $n-5$ \\
$x$ $2$ $n$ & & $x$ $i+1$ $n$& & $x$ $n$ $n-3$\\ \hline
\end{tabular}
\end{center}

The above tables also give us the following $n/2$ subtrails $\{B_1,\dots,B_{\frac{n}{2}}\}$ in $T_y^\prime$.
\begin{align*}
&B_1=x,2,n,n-1,x;\\
&\quad\quad\cdots\\
&B_{\frac{i+1}{2}}=x,i+1,n,G_i(1),G_i(2),\dots, G_i(\tfrac{i-1}{2}),n-i,x \qquad (i \text{ is odd, }3\leq i\leq n-3);\\
&\quad\quad\cdots\\
&B_{\frac{n}{2}}=x,n,n-3,G_{n-1}(1), G_{n-1}(2),\dots,G_{n-1}(\tfrac{n-4}{2}),3,2,1,x.
\end{align*}
where $G_i(j)$ ($3\leq i\leq n-3$) is a subtrail of length $4$ such that $G_i(j)=i-2j,n+1-2j,i+1-2j,n-2j$, and $G_{n-1}(j)$ is a subtrail of length $3$ where $G_{n-1}(j)=n+1-2j,n-2j,n-2j-3$.

Finally we are going to combine those subtrails of $T_x^\prime$ and $T_y^\prime$. Note that any combination will give us an Eulerian circuit on $K_{n+2}-x$ or $K_{n+2}-y$ (respectively), since if $ab$ (or $ba$) appears twice in the subtrails of $T_x^\prime$, then either $T_a^\prime$ or $T_b^\prime$ is not an Eulerian circuit. We let $T_x^\prime=A_1,A_2,\dots,A_{\frac{n}{2}}$ and $T_y^\prime=B_1,B_2,\dots,B_{\frac{n}{2}}$, that means the construction is the following:
\begin{align*}
  &T_x^\prime = y,1,\dots,n,y,3,\dots,n-2,y,5,\dots,4,y,n-1,\dots,2;     \\
  &T_y^\prime = x,2,\dots,n-1,x,4,\dots,5,x,n-2,\dots,3,x,n,\dots,1.
\end{align*}

It remains to show that $T_x^\prime$ and $T_y^\prime$ are strongly compatible. This is true because for every odd vertex $3\leq i\leq n-1$, we can see that $n+3-i,y,i$ is a transition in $T_x^\prime$ and $i,x,n+3-i$ is a transition in $T_y^\prime$, as well as $2,y,1$ is a transition in $T_x^\prime$ and $1,x,2$ is a transition in $T_y^\prime$. This completes the proof.
\end{proof}

\begin{lemma}
\label{lem:K6}
The non-orientable genus of $K_6^3$ is $6$.
\end{lemma}

\begin{proof}
By Lemma \ref{lem:hypergenus}, we have $\ng(K_6^3)\geq 6$. Then it suffices to provide a construction of an embedding of $K_6^3$ in some non-orientable surfaces of genus $6$. By Theorem \ref{thm:embed}, we only need to provide an embedding set which is not strong. The description of such an embedding set $\mathcal{E}_6$ is given below:
\begin{align*}
  &T_1: 4,2,5,3,6,4,5,6,2,3; \\
  &T_2: 4,6,5,1,4,3,1,6,3,5; \\
  &T_3: 1,2,4,6,1,5,2,6,5,4; \\
  &T_4: 5,1,6,2,5,6,3,2,1,3; \\
  &T_5: 6,3,4,2,3,1,2,6,4,1; \\
  &T_6: 2,1,5,3,2,5,4,3,1,4.
\end{align*}

Note that $6,3,4$ is a transition in $T_5$ and $6,5,4$ is a transition in $T_3$, and also $2,3,1$ is a transition in $T_5$ and $1,5,2$ is a transition in $T_3$. That means, neither $T_3$ nor $T_3^{-1}$ is strongly compatible with $T_5$. It is not hard to see that all pairs in the set $\mathcal{E}_6$ are compatible.
\end{proof}

\begin{theorem}
\label{thm:non}
If $n\geq6$ is even, then
\[
\ng(\K)=\frac{(n-2)(n+3)(n-4)}{12}.
\]
\end{theorem}

\begin{proof}
The proof follows the same inductive construction we used in the proof of Theorem \ref{thm:ori}. Instead of using $K_4^3$ as the base step, we use Lemma \ref{lem:K6} as the base. Therefore, by the way we constructed the embedding set of $K_n^3$, Eulerian circuits $T_3$ and $T_5$ will always be compatible, but they will never be strongly compatible.
\end{proof}

\section{Number of non-isomorphic embeddings}

We say that two embeddings $\phi_1,\phi_2: G\to S$ are \emph{isomorphic} if there is an automorphism $\alpha$ of $G$ such that the embeddings $\phi_1$ and $\phi_2\alpha$ are equivalent.
In this section we will show how to obtain many non-isomorphic minimum genus embeddings of $\K$ when $n$ is even.

The number of non-equivalent (2-cell) embeddings of $\K$ in some surface is equal to
\begin{equation}\label{eq:x}
   2^{\binom{n}{3}} \left( \left( \binom{n-1}{2}-1 \right)! \right)^n 2^{2\binom{n}{2}-n+1}
   = 2^{(1-o(1))n^3\log n}.
\end{equation}
This follows from the fact that non-equivalent 2-cell embeddings correspond to different rotation systems. Each vertex in $X_n$ has degree $\binom{n-1}{2}$ and thus has $\big(\binom{n-1}{2}-1\big)!$ possible rotations, and each vertex in $Y_n$ has degree $3$ and thus $2$ possible rotations. The last factor in (\ref{eq:x}) corresponds to the number of inequivalent signatures (on each edge outside a fixed spanning tree we can select the signature freely).

The genera of all these embeddings take only $O(n^3)$ different values, but the majority of them will have their genus much larger than the minimum possible genus. The number of minimum genus embeddings is indeed much smaller as made explicit in the following.

\begin{lemma}
The number of non-equivalent embeddings of $\K$ into a surface of Euler genus $\tfrac{1}{6}(n-2)(n+3)(n-4)$ is at most $2^{(\frac{1}{4}+o(1))n^3\log n}$, where the logarithm is taken base $2$.
\end{lemma}

\begin{proof}
Note that an embedding of $\K$ into a surface of genus $\tfrac{1}{6}(n-2)(n+3)(n-4)$ is a quadrangular embedding of $L_n$. When $n$ is odd, $K_{n-1}$ does not have an Eulerian trail, which implies $L_n$ does not have quadrangular embeddings.
Thus we may assume that $n$ is even since otherwise there are no such embeddings.
By Theorem \ref{thm:embed}, minimum genus embeddings of $\K$ into surfaces of Euler genus $\tfrac{1}{6}(n-2)(n+3)(n-4)$ are quadrilateral and are in a bijective correspondence with embedding sets. These are sets of Eulerian circuits satisfying compatibility conditions. Their number can be estimated as follows.

Suppose that compatible Eulerian circuits $T_1,\dots,T_{k-1}$ are already chosen ($1\le k\le n$). To construct the next circuit $T_k$, we start by an arbitrary edge in $K_n-k$. If we come to a vertex $i < k$ when following the last chosen edge, the transition is determined by compatibility with $T_i$. On the other hand if we come to a vertex $i>k$ for the $r$th time, there are (at most) $n-1-2r$ edges which can be chosen as the next edge on the trail. All together, when passing through such a vertex $i$, we have at most $(n-3)(n-5)(n-7)\cdots 3\cdot 1 = (n-3)!!$ choices. Therefore the number of ways to choose $T_k$ is at most $((n-3)!!)^{n-k}$. Thus the number of embedding sets is at most:
$$
   ((n-3)!!)^{(n-1)+(n-2)+\cdots+1+0} = 2^{(\frac{1}{4}+o(1))n^3 \log n}
$$
and this completes the proof.
\end{proof}

Note that in the proof we are actually giving a bound on compatible closed trail decompositions. Nevertheless, this estimate may be rather tight, since the number of Eulerian circuits in $K_{n-1}$ is $2^{(\frac{1}{2}+o(1))n^2\log n}$, see \cite[Theorem 4]{MR}.

Now we will turn to a lower bound on the number of non-isomorphic minimum genus embeddings that can be obtained by a simple generalization of the construction in our proofs of Theorems \ref{thm:ori} and \ref{thm:non}.

\begin{theorem}
If $n$ is even, there exist at least $2^{(\frac{1}{4}-o(1))n^2\log n}$ non-isomorphic minimum genus embeddings of $\K$ in each, the orientable and the non-orientable surface of Euler genus $\tfrac{1}{6}(n-2)(n+3)(n-4)$.
\end{theorem}

\begin{proof}
Let $I_n$ be the number of non-isomorphic minimum genus embeddings of $\K$. Here we will only deal with the orientable case; for the non-orientable embeddings, arguments are the same.

Recall that in the construction of the embedding set $\E_n=\{T_1^\prime,\dots,T_n^\prime\}$ of $\K$, for every odd vertex $i\in[n-2]$ we arbitrarily pick a transition $a_i,i+1,b_i$ in $T_i\in\E_{n-2}$, and insert a subtrail $E_i$. Since $i+1$ appears exactly $\frac{n-4}{2}$ times in $T_i$, different choice of transitions through $i+1$ will give us different trails $T_i^\prime$ and $T_{i+1}^\prime$. Also, the choice of consecutive odd-even pairs $i,i+1$ gives us a perfect matching of $K_{n-2}$. It is easy to see that any perfect matching of $K_{n-2}$ can be used as such a pairing and this will give us different embedding sets $\E_n$. Moreover, fixing a perfect matching, for example, $i,i+1$ for every odd $i$, we can exchange $i$ and $i+1$ to get a new embedding set. Note that $x$ and $y$ are symmetric in our construction, and can be exchanged.

Let $R_n$ denote the resulting number of inequivalent embedding sets. Then we have:
$$
   R_n \geq \frac{1}{2} \left(\frac{n-4}{2}\right)^{\frac{n-2}{2}}(n-3)!!\ 2^{\frac{n-2}{2}}\ R_{n-2}.
$$
Therefore,
\begin{align*}
\log R_n & \geq\log\prod_{k=2}^{\frac{n-2}{2}} (k-1)^k(2k-1)!!\ 2^{k-1}\\
&\geq\log\frac{(\frac{n-4}{2})!^{\frac{n-2}{2}} 2^{\frac{n}{2}(\frac{n}{2}-2)} \prod_{k=2}^{\frac{n-4}{2}}k!}{\prod_{k=1}^{\frac{n-6}{2}}k!}\\[2mm]
&=\frac{n(n-4)}{4}\,\log\frac{n-4}{2e}+ O(n^2) \\[2mm]
&= (\tfrac{1}{4}-o(1))\, n^2 \log n\, .
\end{align*}
That means that there are at least $2^{(\frac{1}{4}-o(1))n^2\log n}$ inequivalent minimum genus embeddings.

If $\phi_1$ and $\phi_2$ are non-equivalent but isomorphic embeddings, then there is an automorphism $\alpha$ such that $\phi_1$ is equivalent with $\phi_2\alpha$. Each such automorphism is determined by the values $\alpha(i)$ ($i\in [n]$), i.e., by the permutation $\alpha|_{X_n}$ of order $n$. This means that there are at most $n!$ embeddings that are isomorphic with $\phi_1$. 
Thus, the number $I_n$ of isomorphism classes of embeddings is at least $R_n/n!$. Since $\log(n!) = (1+o(1))n\log n$, the denominator in the lower bound on $\log I_n$ decreases the value of $\log R_n$ insignificantly, and thus $\log I_n \ge (\frac{1}{4}-o(1))n^2\log n$.
This completes the proof.
\end{proof}

\section{Hypergraphs with multiple edges}

In this section we are going to investigate the genus of complete $3$-uniform graphs with multiple edges. These results will partially answer the question the authors asked in \cite{JM}. In that work, the genus of random bipartite graphs $\mathcal{G}_{n_1,n_2,p}$ is considered, where $n_1\gg1$ and $n_2$ is a constant, and the edge probabilities are $p=\Theta(n_1^{-1/3})$. In that regime, the following hypergraph occurs. Let $m\K$ be the complete $3$-uniform hypergraph where each triple occurs $m$ times, i.e., each edge of $\K$ has multiplicity $m$.
In this situation, each trail $T_i^m$ in the embedding set $\E_n^m$ is an Eulerian circuit in $m(K_n-i)$. Similarly, we say two trails $T_i$ and $T_j$ are {\em strongly compatible} ({\em compatible}) if transitions $ajb$ appear in $T_i$ exactly $t$ times, then transitions $bia$ ($aib$ or $bia$) appear in $T_j$ exactly $t$ times. It is easy to see that Theorem \ref{thm:embed} is still true in this case, the proof is similar and we omit the details. Therefore, we have the following result.

\begin{theorem}\label{thm:even}
If $n\ge4$ is even and $m\ge2$, then $\g(m\K)=\frac{(n-2)(mn(n-1)-12)}{24}$ and $\ng(m\K)=\frac{(n-2)(mn(n-1)-12)}{12}$.
\end{theorem}

\begin{figure}[htb]
\centering
\includegraphics[width=2in]{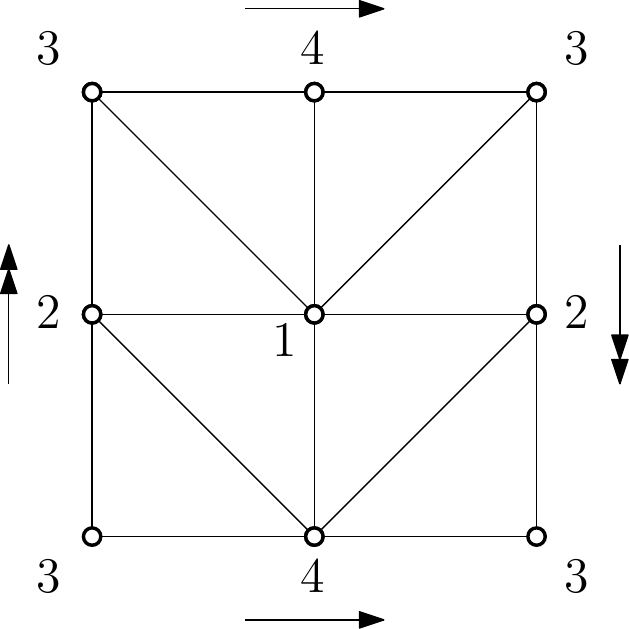}
\caption{A non-orientable embedding of $2K_4^3$ on the Klein bottle.}\label{fig:2K_4}
\end{figure}

\begin{proof}
The lower bound follows by Lemma \ref{lem:hypergenus}. To see the upper bound, we will give an inductive construction on $m$.

Suppose $\E_n^m=\{T_1^m,\dots,T_n^m\}$ is a (strong) embedding set of $m\K$, and suppose that $\E_n=\{T_1,\dots,T_n\}$ is a (strong) embedding set of $\K$. For every odd $i\in[n]$, suppose the transition $a_i,i+1,b_i$ is in both $T_i^m$ and $T_i$. Note that by our construction, such transition exists, and actually we have at least $\frac{n-2}{2}$ such transitions for every $i$. We arbitrarily pick one such transition $a_i,i+1,b_i$. Since $T_i$ also contains the transition $a_i,i+1,b_i$, we break $T_i$ between $i+1$ and $b_i$, and we write $T_i$ by starting with $b_i$ and end with $a_i,i+1$. We also break transition $a_i,i+1,b_i$ in $T_i^m$, and insert $T_i$ between $i+1$ and $b_i$. For the case $i+1$, we do the same things on transition $b_i,i,a_i$. Therefore, we will get a (strong) embedding set $\E_n^{m+1}$ of $(m+1)\K$. It is easy to verify the (strong) compatibility among trails in $\E_n^{m+1}$.

The described construction works in all cases except when $n=4$, and we look for the non-orientable embeddings of $mK_4^3$. Since $\g(K_4^3)=0$, the base case of induction for the non-orientable genus of $mK_4^3$ is when $m=2$. In this case, we construct the following non-orientable embedding set $\E_4^2$ on the Klein bottle (see Figure \ref{fig:2K_4} for the corresponding embedding):
\begin{align*}
  &T_1^2 : 3,2,4,2,3,4; \\
  &T_2^2 : 4,1,3,4,1,3; \\
  &T_3^2 : 1,4,2,1,4,2; \\
  &T_4^2 : 2,3,1,3,2,1.
\end{align*}
The induction step follows the same argument as when $n\geq6$.
\end{proof}

Let us observe that the embedding of $mK_n^3$ described in the proof of Theorem \ref{thm:even} (with the exception of the non-orientable case when $n=4$) is just a branched covering from a quadrilateral embedding of $K_n^3$ where each vertex $i\in X_n$ is a branch point with branching degree $m$.

\bibliographystyle{abbrv}
\bibliography{reference}

\end{document}